\documentclass{amsart}
\usepackage{amsmath}
\usepackage{amsfonts}

\newtheorem{theorem}{Theorem}
\theoremstyle{plain}

\newtheorem{corollary}{Corollary}

\newtheorem{lemma}{Lemma}

\newtheorem{remark}{Remark}

\numberwithin{equation}{section}
\input{tcilatex}

\begin{document}
\title[Integral inequalities]{New estimates on generalization of some
integral inequalities for $(\alpha ,m)-$convex functions}
\author{\.{I}mdat \.{I}\c{s}can}
\address{Department of Mathematics, Faculty of Arts and Sciences,\\
Giresun University, 28100, Giresun, Turkey.}
\email{imdat.iscan@giresun.edu.tr}
\date{July 01, 2012}

\begin{abstract}
In this paper, we derive new estimates for the remainder term of the
midpoint, trapezoid, and Simpson formulae for functions whose derivatives in
absolute value at certain power are $(\alpha ,m)-$convex.
\end{abstract}

\maketitle

\section{Introduction}

Let $f:I\subseteq \mathbb{R\rightarrow R}$ be a convex function defined on
the interval $I$ of real numbers and $a,b\in I$ with $a<b$. The following
double inequality is well known in the literature as Hermite-Hadamard
integral inequality 
\begin{equation}
f\left( \frac{a+b}{2}\right) \leq \frac{1}{b-a}\dint\limits_{a}^{b}f(x)dx%
\leq \frac{f(a)+f(b)}{2}\text{.}  \label{1-1}
\end{equation}

The class of $(\alpha ,m)-$convex functions was first introduced In \cite%
{M93}, and it is defined as follows:

The function $f:\left[ 0,b\right] \mathbb{\rightarrow R}$, $b>0$, is said to
be $(\alpha ,m)$-convex where $(\alpha ,m)\in \left[ 0,1\right] ^{2}$, if we
have$f$%
\begin{equation*}
\left( tx+m(1-t)y\right) \leq t^{\alpha }f(x)+m(1-t^{\alpha })f(y)
\end{equation*}

for all $x,y\in \left[ 0,b\right] $ and $t\in \left[ 0,1\right] $.

It can be easily that for $(\alpha ,m)\in \left\{ (0,0),(\alpha
,0),(1,0),(1,m),(1,1),(\alpha ,1)\right\} $ one obtains the following
classes of functions: increasing, $\alpha $-starshaped, starshaped, $m$%
-convex, convex, $\alpha $-convex.

Denote by $K_{m}^{\alpha }(b)$ the set of all $(\alpha ,m)$-convex functions
on $\left[ 0,b\right] $ for which $f(0)\leq 0$. For recent results and
generalizations concerning $(\alpha ,m)$-convex functions (see \cite%
{BOP08,M93,OAK11,OKS10,OSS11,SSOR09}).

The following inequality is well known in the literature as Simpson's
inequality .

Let $f:\left[ a,b\right] \mathbb{\rightarrow R}$ be a four times
continuously differentiable mapping on $\left( a,b\right) $ and $\left\Vert
f^{(4)}\right\Vert _{\infty }=\underset{x\in \left( a,b\right) }{\sup }%
\left\vert f^{(4)}(x)\right\vert <\infty .$ Then the following inequality
holds:%
\begin{equation*}
\left\vert \frac{1}{3}\left[ \frac{f(a)+f(b)}{2}+2f\left( \frac{a+b}{2}%
\right) \right] -\frac{1}{b-a}\dint\limits_{a}^{b}f(x)dx\right\vert \leq 
\frac{1}{2880}\left\Vert f^{(4)}\right\Vert _{\infty }\left( b-a\right) ^{2}.
\end{equation*}%
\ \qquad In recent years many authors have studied error estimations for
Simpson's inequality; for refinements, counterparts, generalizations and new
Simpson's type inequalities, see \cite{P12,SA11,SSO10,SSO10a}.

In this paper, in order to provide a unified approach to establish midpoint
inequality, trapezoid inequality and Simpson's inequality for functions
whose derivatives in absolute value at certain power are $(\alpha ,m)$%
-convex, we derive a general integral identity for convex functions.

\section{Main results}

In order to generalize the classical Trapezoid, midpoint and Simpson type
inequalities and prove them, we need the following Lemma:

\begin{lemma}
\label{2.1}Let $f:I\subset \mathbb{R\rightarrow R}$ be a differentiable
mapping on $I^{\circ }$ such that $f^{\prime }\in L[ma,mb]$, where $m\in
\left( 0,1\right] $, $ma,mb\in I$ with $a<b$, then for $\theta ,\lambda \in %
\left[ 0,1\right] $ the following equality holds:%
\begin{equation*}
\left( 1-\theta \right) \left( \lambda f(ma)+\left( 1-\lambda \right)
f(mb)\right) +\theta f(\left( 1-\lambda \right) ma+\lambda mb)-\frac{1}{%
m\left( b-a\right) }\dint\limits_{ma}^{mb}f(x)dx
\end{equation*}%
\begin{eqnarray}
&=&m\left( b-a\right) \left[ -\lambda ^{2}\dint\limits_{0}^{1}\left(
t-\theta \right) f^{\prime }\left( tma+\left( 1-t\right) \left[ \left(
1-\lambda \right) ma+\lambda mb\right] \right) dt\right.  \notag \\
&&\left. +\left( 1-\lambda \right) ^{2}\dint\limits_{0}^{1}\left( t-\theta
\right) f^{\prime }\left( tmb+\left( 1-t\right) \left[ \left( 1-\lambda
\right) ma+\lambda mb\right] \right) dt\right] .  \notag
\end{eqnarray}
\end{lemma}

A simple proof of the equality can be done by performing an integration by
parts in the integrals from the right side and changing the variable. The
details are left to the interested reader.

\begin{theorem}
\label{2.2}Let $f:I\subseteq \lbrack 0,\infty )\rightarrow \mathbb{R}$ be a
differentiable mapping on $I^{\circ }$ such that $f^{\prime }\in L[ma,mb]$,
where $m\in \left( 0,1\right] $, $ma,mb\in I^{\circ }$ with $a<b$ and $%
\theta ,\lambda \in \left[ 0,1\right] $. If $\left\vert f^{\prime
}\right\vert ^{q}$ is $(\alpha ,m)-$convex on $[ma,mb]$, for $\alpha \in %
\left[ 0,1\right] $, $q\geq 1$ then the following inequality holds:%
\begin{eqnarray}
&&\left\vert \left( 1-\theta \right) \left( \lambda f(ma)+\left( 1-\lambda
\right) f(mb)\right) +\theta f(mC)-\frac{1}{m\left( b-a\right) }%
\dint\limits_{ma}^{mb}f(x)dx\right\vert   \notag \\
&\leq &m\left( b-a\right) A_{1}^{1-\frac{1}{q}}(\theta )\min \left\{
B_{1}\left( \theta ,\lambda ,\alpha ,q,m\right) ,B_{2}\left( \theta ,\lambda
,\alpha ,q,m\right) \right\} ,  \label{2-2}
\end{eqnarray}%
where%
\begin{eqnarray*}
B_{1}\left( \theta ,\lambda ,\alpha ,q,m\right)  &=&\left\{ \lambda
^{2}\left( \left\vert f^{\prime }(ma)\right\vert ^{q}A_{2}(\theta ,\alpha
)+m\left\vert f^{\prime }(C)\right\vert ^{q}A_{3}(\theta ,\alpha )\right) ^{%
\frac{1}{q}}\right.  \\
&&\left. +\left( 1-\lambda \right) ^{2}\left( \left\vert f^{\prime
}(mb)\right\vert ^{q}A_{2}(\theta ,\alpha )+m\left\vert f^{\prime
}(C)\right\vert ^{q}A_{3}(\theta ,\alpha )\right) ^{\frac{1}{q}}\right\} , \\
B_{2}\left( \theta ,\lambda ,\alpha ,q,m\right)  &=&\left\{ \lambda
^{2}\left( \left\vert f^{\prime }(mC)\right\vert ^{q}A_{4}(\theta ,\alpha
)+m\left\vert f^{\prime }(a)\right\vert ^{q}A_{5}(\theta ,\alpha )\right) ^{%
\frac{1}{q}}\right.  \\
&&\left. +\left( 1-\lambda \right) ^{2}\left( \left\vert f^{\prime
}(mC)\right\vert ^{q}A_{4}(\theta ,\alpha )+m\left\vert f^{\prime
}(b)\right\vert ^{q}A_{5}(\theta ,\alpha )\right) ^{\frac{1}{q}}\right\} ,
\end{eqnarray*}%
\begin{eqnarray*}
A_{1}(\theta ) &=&\theta ^{2}-\theta +\frac{1}{2}, \\
A_{2}(\theta ,\alpha ) &=&\frac{2\theta ^{\alpha +2}}{\left( \alpha
+1\right) \left( \alpha +2\right) }-\frac{\theta }{\alpha +1}+\frac{1}{%
\alpha +2}, \\
A_{3}(\theta ,\alpha ) &=&\theta ^{2}-\frac{2\theta ^{\alpha +2}}{\left(
\alpha +1\right) \left( \alpha +2\right) }-\frac{\alpha \theta }{\alpha +1}+%
\frac{\alpha }{2\left( \alpha +2\right) }, \\
A_{4}(\theta ,\alpha ) &=&\frac{2\left( 1-\theta \right) ^{\alpha +2}}{%
\left( \alpha +1\right) \left( \alpha +2\right) }-\frac{1-\theta }{\alpha +1}%
+\frac{1}{\alpha +2}, \\
A_{5}(\theta ,\alpha ) &=&\left( 1-\theta \right) ^{2}-\frac{2\left(
1-\theta \right) ^{\alpha +2}}{\left( \alpha +1\right) \left( \alpha
+2\right) }-\frac{\alpha \left( 1-\theta \right) }{\alpha +1}+\frac{\alpha }{%
2\left( \alpha +2\right) },
\end{eqnarray*}%
and $C=\left( 1-\lambda \right) a+\lambda b.$
\end{theorem}

\begin{proof}
Suppose that $C=\left( 1-\lambda \right) a+\lambda b.$ From Lemma \ref{2.1}
and using the properties of modulus and the well known power mean
inequality, we have%
\begin{eqnarray*}
&&\left( 1-\theta \right) \left( \lambda f(ma)+\left( 1-\lambda \right)
f(mb)\right) +\theta f(\left( 1-\lambda \right) ma+\lambda mb)-\frac{1}{%
m\left( b-a\right) }\dint\limits_{ma}^{mb}f(x)dx \\
&\leq &m\left( b-a\right) \left[ \lambda ^{2}\dint\limits_{0}^{1}\left\vert
t-\theta \right\vert \left\vert f^{\prime }\left( tma+\left( 1-t\right)
mC\right) \right\vert dt\right. \\
&&\left. +\left( 1-\lambda \right) ^{2}\dint\limits_{0}^{1}\left\vert
t-\theta \right\vert \left\vert f^{\prime }\left( tmb+\left( 1-t\right)
mC\right) \right\vert dt\right] \\
&\leq &m\left( b-a\right) \left\{ \lambda ^{2}\left(
\dint\limits_{0}^{1}\left\vert t-\theta \right\vert dt\right) ^{1-\frac{1}{q}%
}\left( \dint\limits_{0}^{1}\left\vert t-\theta \right\vert \left\vert
f^{\prime }\left( tma+\left( 1-t\right) mC\right) \right\vert ^{q}dt\right)
^{\frac{1}{q}}\right.
\end{eqnarray*}%
\begin{equation}
\left. +\left( 1-\lambda \right) ^{2}\left( \dint\limits_{0}^{1}\left\vert
t-\theta \right\vert dt\right) ^{1-\frac{1}{q}}\left(
\dint\limits_{0}^{1}\left\vert t-\theta \right\vert \left\vert f^{\prime
}\left( tmb+\left( 1-t\right) mC\right) \right\vert ^{q}dt\right) ^{\frac{1}{%
q}}\right\}  \label{2-2a}
\end{equation}%
Since $\left\vert f^{\prime }\right\vert ^{q}$ is $(\alpha ,m)-$convex on $%
[a,b],$ we know that for $t\in \left[ 0,1\right] $%
\begin{equation*}
\left\vert f^{\prime }\left( tma+\left( 1-t\right) mC\right) \right\vert
^{q}\leq t^{\alpha }\left\vert f^{\prime }(ma)\right\vert ^{q}+m(1-t^{\alpha
})\left\vert f^{\prime }(C)\right\vert ^{q},
\end{equation*}%
and 
\begin{equation*}
\left\vert f^{\prime }\left( tmb+\left( 1-t\right) mC\right) \right\vert
^{q}\leq t^{\alpha }\left\vert f^{\prime }(mb)\right\vert ^{q}+m(1-t^{\alpha
})\left\vert f^{\prime }(C)\right\vert ^{q}.
\end{equation*}%
Hence, by simple computation%
\begin{eqnarray}
&&\dint\limits_{0}^{1}\left\vert t-\theta \right\vert t^{\alpha }\left\vert
f^{\prime }(ma)\right\vert ^{q}+m(1-t^{\alpha })\left\vert f^{\prime
}(C)\right\vert ^{q}dt  \label{2-2b} \\
&=&\left\vert f^{\prime }(ma)\right\vert ^{q}A_{2}(\theta ,\alpha
)+m\left\vert f^{\prime }(C)\right\vert ^{q}A_{3}(\theta ,\alpha )  \notag
\end{eqnarray}%
\begin{eqnarray}
&&\dint\limits_{0}^{1}\left\vert t-\theta \right\vert t^{\alpha }\left\vert
f^{\prime }(mb)\right\vert ^{q}+m(1-t^{\alpha })\left\vert f^{\prime
}(C)\right\vert ^{q}dt  \label{2-2c} \\
&=&\left\vert f^{\prime }(mb)\right\vert ^{q}A_{2}(\theta ,\alpha
)+m\left\vert f^{\prime }(C)\right\vert ^{q}A_{3}(\theta ,\alpha )  \notag
\end{eqnarray}%
and%
\begin{equation}
\dint\limits_{0}^{1}\left\vert t-\theta \right\vert dt=\theta ^{2}-\theta +%
\frac{1}{2}.  \label{2-2d}
\end{equation}%
Thus, using (\ref{2-2b})-(\ref{2-2d}) in (\ref{2-2a}), we obtain the
following inequality 
\begin{eqnarray}
&&\left\vert \left( 1-\theta \right) \left( \lambda f(ma)+\left( 1-\lambda
\right) f(mb)\right) +\theta f(mC)-\frac{1}{m\left( b-a\right) }%
\dint\limits_{ma}^{mb}f(x)dx\right\vert  \notag \\
&\leq &m\left( b-a\right) A_{1}^{1-\frac{1}{q}}(\theta )\left\{ \left(
\left\vert f^{\prime }(ma)\right\vert ^{q}A_{2}(\theta ,\alpha )+m\left\vert
f^{\prime }(C)\right\vert ^{q}A_{3}(\theta ,\alpha )\right) ^{\frac{1}{q}%
}\right.  \notag \\
&&\left. +\left( \left\vert f^{\prime }(mb)\right\vert ^{q}A_{2}(\theta
,\alpha )+m\left\vert f^{\prime }(C)\right\vert ^{q}A_{3}(\theta ,\alpha
)\right) ^{\frac{1}{q}}\right\} .  \label{2-2e}
\end{eqnarray}

In the inequality (\ref{2-2a}), if we use equalities 
\begin{equation*}
\dint\limits_{0}^{1}\left\vert t-\theta \right\vert \left\vert f^{\prime
}\left( tma+\left( 1-t\right) mC\right) \right\vert
^{q}dt=\dint\limits_{0}^{1}\left\vert 1-\theta -t\right\vert \left\vert
f^{\prime }\left( tmC+\left( 1-t\right) ma\right) \right\vert ^{q}dt
\end{equation*}%
and%
\begin{equation*}
\dint\limits_{0}^{1}\left\vert t-\theta \right\vert \left\vert f^{\prime
}\left( tmb+\left( 1-t\right) mC\right) \right\vert
^{q}dt=\dint\limits_{0}^{1}\left\vert 1-\theta -t\right\vert \left\vert
f^{\prime }\left( tmC+\left( 1-t\right) mb\right) \right\vert ^{q}dt,
\end{equation*}%
by similar process, since $\left\vert f^{\prime }\right\vert ^{q}$ is $%
(\alpha ,m)-$convex on $[a,b],$ for $t\in \left[ 0,1\right] $%
\begin{equation*}
\left\vert f^{\prime }\left( tmC+\left( 1-t\right) ma\right) \right\vert
^{q}\leq t^{\alpha }\left\vert f^{\prime }(mC)\right\vert ^{q}+m(1-t^{\alpha
})\left\vert f^{\prime }(a)\right\vert ^{q}
\end{equation*}%
and%
\begin{equation*}
\left\vert f^{\prime }\left( tmC+\left( 1-t\right) mb\right) \right\vert
^{q}\leq t^{\alpha }\left\vert f^{\prime }(mC)\right\vert ^{q}+m(1-t^{\alpha
})\left\vert f^{\prime }(b)\right\vert ^{q}.
\end{equation*}%
Similarly, by simple computation%
\begin{eqnarray}
&&\dint\limits_{0}^{1}\left\vert 1-\theta -t\right\vert t^{\alpha
}\left\vert f^{\prime }(mC)\right\vert ^{q}+m(1-t^{\alpha })\left\vert
f^{\prime }(a)\right\vert ^{q}dt  \label{2-2f} \\
&=&\left\vert f^{\prime }(mC)\right\vert ^{q}A_{4}(\theta ,\alpha
)+m\left\vert f^{\prime }(a)\right\vert ^{q}A_{5}(\theta ,\alpha )  \notag
\end{eqnarray}%
\begin{eqnarray}
&&\dint\limits_{0}^{1}\left\vert 1-\theta -t\right\vert t^{\alpha
}\left\vert f^{\prime }(mC)\right\vert ^{q}+m(1-t^{\alpha })\left\vert
f^{\prime }(b)\right\vert ^{q}dt  \label{2-2g} \\
&=&\left\vert f^{\prime }(mC)\right\vert ^{q}A_{4}(\theta ,\alpha
)+m\left\vert f^{\prime }(b)\right\vert ^{q}A_{5}(\theta ,\alpha )  \notag
\end{eqnarray}%
Thus, using (\ref{2-2b}),(\ref{2-2f}) and (\ref{2-2g}) in (\ref{2-2a}), we
have the following inequality%
\begin{eqnarray}
&&\left\vert \left( 1-\theta \right) \left( \lambda f(ma)+\left( 1-\lambda
\right) f(mb)\right) +\theta f(mC)-\frac{1}{m\left( b-a\right) }%
\dint\limits_{ma}^{mb}f(x)dx\right\vert  \notag \\
&\leq &m\left( b-a\right) A_{1}^{1-\frac{1}{q}}(\theta )\left\{ \left(
\left\vert f^{\prime }(mC)\right\vert ^{q}A_{4}(\theta ,\alpha )+m\left\vert
f^{\prime }(a)\right\vert ^{q}A_{5}(\theta ,\alpha )\right) ^{\frac{1}{q}%
}\right.  \notag \\
&&\left. +\left( \left\vert f^{\prime }(mC)\right\vert ^{q}A_{4}(\theta
,\alpha )+m\left\vert f^{\prime }(b)\right\vert ^{q}A_{5}(\theta ,\alpha
)\right) ^{\frac{1}{q}}\right\} .  \label{2-2h}
\end{eqnarray}

From the inequalities (\ref{2-2e}) and (\ref{2-2h}) \ the inequality (\ref%
{2-2}) is obtained. This completes the proof.
\end{proof}

\begin{corollary}
\label{2.3}Under the assumptions of Theorem \ref{2.2} with $q=1$%
\begin{eqnarray*}
&&\left\vert \left( 1-\theta \right) \left( \lambda f(ma)+\left( 1-\lambda
\right) f(mb)\right) +\theta f(mC)-\frac{1}{m\left( b-a\right) }%
\dint\limits_{ma}^{mb}f(x)dx\right\vert  \\
&\leq &m\left( b-a\right) \min \left\{ B_{1}\left( \theta ,\lambda ,\alpha
,1,m\right) ,B_{2}\left( \theta ,\lambda ,\alpha ,1,m\right) \right\} .
\end{eqnarray*}
\end{corollary}

\begin{corollary}
\label{2.3a}Under the assumptions of Theorem \ref{2.2} with $\lambda =\frac{1%
}{2}$ and $\theta =\frac{2}{3}$, we have%
\begin{eqnarray*}
&&\left\vert \frac{1}{6}\left[ f(ma)+4f\left( \frac{ma+mb}{2}\right) +f(mb)%
\right] -\frac{1}{m\left( b-a\right) }\dint\limits_{ma}^{mb}f(x)dx\right%
\vert  \\
&\leq &m\left( b-a\right) \left( \frac{5}{18}\right) ^{1-\frac{1}{q}}\min
\left\{ B_{1}\left( \frac{2}{3},\frac{1}{2},\alpha ,q,m\right) ,B_{2}\left( 
\frac{2}{3},\frac{1}{2},\alpha ,q,m\right) \right\} .
\end{eqnarray*}
\end{corollary}

\begin{remark}
In Corollary \ref{2.3a}, if we take $\alpha =m=1,$ we obtain the following
inequality%
\begin{eqnarray*}
&&\left\vert \frac{1}{6}\left[ f(a)+4f\left( \frac{a+b}{2}\right) +f(b)%
\right] -\frac{1}{b-a}\dint\limits_{a}^{b}f(x)dx\right\vert  \\
&\leq &\left( b-a\right) A_{1}^{1-\frac{1}{q}}(\theta )\min \left\{
B_{1}\left( \frac{2}{3},\frac{1}{2},1,q,1\right) ,B_{2}\left( \frac{2}{3},%
\frac{1}{2},1,q,1\right) \right\}  \\
&\leq &\left( b-a\right) A_{1}^{1-\frac{1}{q}}(\theta )B_{2}\left( \frac{2}{3%
},\frac{1}{2},1,q,1\right)  \\
&\leq &\left( b-a\right) \left( \frac{5}{72}\right) ^{1-\frac{1}{q}}\left\{
\left( \frac{29}{648}\left\vert f^{\prime }(\frac{a+b}{2})\right\vert ^{q}+%
\frac{2}{81}\left\vert f^{\prime }(a)\right\vert ^{q}\right) ^{\frac{1}{q}%
}\right.  \\
&&\left. +\left( \frac{29}{648}\left\vert f^{\prime }(\frac{a+b}{2}%
)\right\vert ^{q}+\frac{2}{81}\left\vert f^{\prime }(b)\right\vert
^{q}\right) ^{\frac{1}{q}}\right\} ,
\end{eqnarray*}%
which is the better than the inequality in \cite[Theorem 10]{SSO10} for $s=1$
.
\end{remark}

\begin{corollary}
Under the assumptions of Theorem \ref{2.2} with $\lambda =\frac{1}{2}$ and $%
\theta =0$, we have%
\begin{eqnarray*}
&&\left\vert \frac{f(ma)+f(mb)}{2}-\frac{1}{m\left( b-a\right) }%
\dint\limits_{ma}^{mb}f(x)dx\right\vert  \\
&\leq &m\left( b-a\right) \left( \frac{1}{2}\right) ^{1-\frac{1}{q}}\min
\left\{ B_{1}\left( 0,\frac{1}{2},\alpha ,q,m\right) ,B_{2}\left( 0,\frac{1}{%
2},\alpha ,q,m\right) \right\} ,
\end{eqnarray*}
\end{corollary}

\begin{corollary}
Under the assumptions of Theorem \ref{2.2} with $\lambda =\frac{1}{2}$ and $%
\theta =1$, we have%
\begin{eqnarray*}
&&\left\vert f\left( \frac{ma+mb}{2}\right) -\frac{1}{m\left( b-a\right) }%
\dint\limits_{ma}^{mb}f(x)dx\right\vert  \\
&\leq &m\left( b-a\right) \left( \frac{1}{2}\right) ^{1-\frac{1}{q}}\min
\left\{ B_{1}\left( 1,\frac{1}{2},\alpha ,q,m\right) ,B_{2}\left( 1,\frac{1}{%
2},\alpha ,q,m\right) \right\} 
\end{eqnarray*}
\end{corollary}

\begin{theorem}
\label{2.4}Let $f:I\subseteq \lbrack 0,\infty )\rightarrow \mathbb{R}$ be a
differentiable mapping on $I^{\circ }$ such that $f^{\prime }\in L[ma,mb]$,
where $m\in \left( 0,1\right] $, $ma,mb\in I^{\circ }$ with $a<b$ and $%
\theta ,\lambda \in \left[ 0,1\right] $. If $\left\vert f^{\prime
}\right\vert ^{q}$ is $(\alpha ,m)-$convex on $[ma,mb]$, for $\alpha \in %
\left[ 0,1\right] $, $q>1$ then the following inequality holds:%
\begin{eqnarray*}
&&\left\vert \left( 1-\theta \right) \left( \lambda f(ma)+\left( 1-\lambda
\right) f(mb)\right) +\theta f(mC)-\frac{1}{m\left( b-a\right) }%
\dint\limits_{ma}^{mb}f(x)dx\right\vert \\
&\leq &m\left( b-a\right) \left( \frac{\theta ^{p+1}+\left( 1-\theta \right)
^{p+1}}{p+1}\right) ^{\frac{1}{p}}\min \left\{ B_{3}\left( \lambda ,\alpha
,q\right) ,B_{4}\left( \lambda ,\alpha ,q\right) \right\} ,
\end{eqnarray*}%
where 
\begin{equation*}
B_{3}\left( \lambda ,\alpha ,q\right) =\left\{ \lambda ^{2}E_{1}^{\frac{1}{q}%
}(\lambda ,\alpha )+\left( 1-\lambda \right) ^{2}E_{2}^{\frac{1}{q}}(\lambda
,\alpha )\right\} ,
\end{equation*}%
\begin{equation*}
B_{4}\left( \lambda ,\alpha ,q\right) =\left\{ \lambda ^{2}E_{3}^{\frac{1}{q}%
}(\lambda ,\alpha )+\left( 1-\lambda \right) ^{2}E_{4}^{\frac{1}{q}}(\lambda
,\alpha )\right\} ,
\end{equation*}%
\begin{eqnarray*}
E_{1}(\lambda ,\alpha ) &=&\frac{\left\vert f^{\prime }\left( ma\right)
\right\vert ^{q}+\alpha m\left\vert f^{\prime }\left( C\right) \right\vert
^{q}}{\alpha +1}, \\
E_{2}(\lambda ,\alpha ) &=&\frac{\left\vert f^{\prime }\left( mb\right)
\right\vert ^{q}+\alpha m\left\vert f^{\prime }\left( C\right) \right\vert
^{q}}{\alpha +1}, \\
E_{3}(\lambda ,\alpha ) &=&\frac{\left\vert f^{\prime }\left( mC\right)
\right\vert ^{q}+\alpha m\left\vert f^{\prime }\left( a\right) \right\vert
^{q}}{\alpha +1}, \\
E_{4}(\lambda ,\alpha ) &=&\frac{\left\vert f^{\prime }\left( mC\right)
\right\vert ^{q}+\alpha m\left\vert f^{\prime }\left( b\right) \right\vert
^{q}}{\alpha +1},
\end{eqnarray*}%
$C=\left( 1-\lambda \right) a+\lambda b$ and $\frac{1}{p}+\frac{1}{q}=1.$
\end{theorem}

\begin{proof}
Suppose that $C=\left( 1-\lambda \right) a+\lambda b.$ From Lemma \ref{2.1}
and by H\"{o}lder's integral inequality, we have%
\begin{eqnarray*}
&&\left\vert \left( 1-\theta \right) \left( \lambda f(ma)+\left( 1-\lambda
\right) f(mb)\right) +\theta f(mC)-\frac{1}{m\left( b-a\right) }%
\dint\limits_{ma}^{mb}f(x)dx\right\vert \leq m\left( b-a\right) \\
&&\left[ \lambda ^{2}\dint\limits_{0}^{1}\left\vert t-\theta \right\vert
\left\vert f^{\prime }\left( tma+m\left( 1-t\right) C\right) \right\vert
dt+\left( 1-\lambda \right) ^{2}\dint\limits_{0}^{1}\left\vert t-\theta
\right\vert \left\vert f^{\prime }\left( tmb+m\left( 1-t\right) C\right)
\right\vert dt\right] \\
&\leq &m\left( b-a\right) \left\{ \lambda ^{2}\left(
\dint\limits_{0}^{1}\left\vert t-\theta \right\vert ^{p}dt\right) ^{\frac{1}{%
p}}\left( \dint\limits_{0}^{1}\left\vert f^{\prime }\left( tma+m\left(
1-t\right) C\right) \right\vert ^{q}dt\right) ^{\frac{1}{q}}\right.
\end{eqnarray*}%
\begin{equation}
\left. +\left( 1-\lambda \right) ^{2}\left( \dint\limits_{0}^{1}\left\vert
t-\theta \right\vert ^{p}dt\right) ^{\frac{1}{p}}\left(
\dint\limits_{0}^{1}\left\vert f^{\prime }\left( tmb+m\left( 1-t\right)
C\right) \right\vert ^{q}dt\right) ^{\frac{1}{q}}\right\} .  \label{2-3a}
\end{equation}%
Since $\left\vert f^{\prime }\right\vert ^{q}$ is $(\alpha ,m)-$convex on $%
[a,b],$ we know that for $t\in \left[ 0,1\right] $%
\begin{equation*}
\left\vert f^{\prime }\left( tma+m\left( 1-t\right) C\right) \right\vert
^{q}\leq t^{\alpha }\left\vert f^{\prime }(ma)\right\vert ^{q}+m(1-t^{\alpha
})\left\vert f^{\prime }(C)\right\vert ^{q},
\end{equation*}%
and 
\begin{equation*}
\left\vert f^{\prime }\left( tmb+m\left( 1-t\right) C\right) \right\vert
^{q}\leq t^{\alpha }\left\vert f^{\prime }(mb)\right\vert ^{q}+m(1-t^{\alpha
})\left\vert f^{\prime }(C)\right\vert ^{q}.
\end{equation*}%
Hence, by simple computation%
\begin{equation}
\dint\limits_{0}^{1}t^{\alpha }\left\vert f^{\prime }(ma)\right\vert
^{q}+m(1-t^{\alpha })\left\vert f^{\prime }(C)\right\vert ^{q}dt=\frac{%
\left\vert f^{\prime }\left( ma\right) \right\vert ^{q}+\alpha m\left\vert
f^{\prime }\left( C\right) \right\vert ^{q}}{\alpha +1},  \label{2-3b}
\end{equation}%
\begin{equation}
\dint\limits_{0}^{1}t^{\alpha }\left\vert f^{\prime }(mb)\right\vert
^{q}+m(1-t^{\alpha })\left\vert f^{\prime }(C)\right\vert ^{q}dt=\frac{%
\left\vert f^{\prime }\left( mb\right) \right\vert ^{q}+\alpha m\left\vert
f^{\prime }\left( C\right) \right\vert ^{q}}{\alpha +1},  \label{2-3c}
\end{equation}%
and%
\begin{equation}
\dint\limits_{0}^{1}\left\vert t-\theta \right\vert ^{p}dt=\frac{\theta
^{p+1}+\left( 1-\theta \right) ^{p+1}}{p+1}  \label{2-3d}
\end{equation}%
thus, using (\ref{2-3b})-(\ref{2-3d}) in (\ref{2-3a}), we obtain the
following inequality%
\begin{eqnarray}
&&\left\vert \left( 1-\theta \right) \left( \lambda f(ma)+\left( 1-\lambda
\right) f(mb)\right) +\theta f(mC)-\frac{1}{m\left( b-a\right) }%
\dint\limits_{ma}^{mb}f(x)dx\right\vert  \notag \\
&\leq &m\left( b-a\right) \left( \frac{\theta ^{p+1}+\left( 1-\theta \right)
^{p+1}}{p+1}\right) ^{\frac{1}{p}}\left\{ \lambda ^{2}\left( \frac{%
\left\vert f^{\prime }\left( ma\right) \right\vert ^{q}+\alpha m\left\vert
f^{\prime }\left( C\right) \right\vert ^{q}}{\alpha +1}\right) ^{\frac{1}{q}%
}\right.  \notag \\
&&\left. \left( 1-\lambda \right) ^{2}\left( \frac{\left\vert f^{\prime
}\left( mb\right) \right\vert ^{q}+\alpha m\left\vert f^{\prime }\left(
C\right) \right\vert ^{q}}{\alpha +1}\right) ^{\frac{1}{q}}\right\} .
\label{2-3e}
\end{eqnarray}%
Similarly%
\begin{eqnarray}
\dint\limits_{0}^{1}\left\vert f^{\prime }\left( tma+\left( 1-t\right)
mC\right) \right\vert ^{q}dt &=&\dint\limits_{0}^{1}\left\vert f^{\prime
}\left( tmC+\left( 1-t\right) ma\right) \right\vert ^{q}dt  \notag \\
&\leq &\dint\limits_{0}^{1}t^{\alpha }\left\vert f^{\prime }(mC)\right\vert
^{q}+m(1-t^{\alpha })\left\vert f^{\prime }(a)\right\vert ^{q}dt  \notag \\
&=&\frac{\left\vert f^{\prime }\left( mC\right) \right\vert ^{q}+\alpha
m\left\vert f^{\prime }\left( a\right) \right\vert ^{q}}{\alpha +1}
\label{2-3f}
\end{eqnarray}%
and%
\begin{eqnarray}
\dint\limits_{0}^{1}\left\vert f^{\prime }\left( tmb+\left( 1-t\right)
mC\right) \right\vert ^{q}dt &=&\dint\limits_{0}^{1}\left\vert f^{\prime
}\left( tmC+\left( 1-t\right) mb\right) \right\vert ^{q}dt  \notag \\
&\leq &\dint\limits_{0}^{1}t^{\alpha }\left\vert f^{\prime }(mC)\right\vert
^{q}+m(1-t^{\alpha })\left\vert f^{\prime }(b)\right\vert ^{q}dt  \notag \\
&=&\frac{\left\vert f^{\prime }\left( mC\right) \right\vert ^{q}+\alpha
m\left\vert f^{\prime }\left( b\right) \right\vert ^{q}}{\alpha +1}.
\label{2-3g}
\end{eqnarray}
By using (\ref{2-3b}), (\ref{2-3f}) and (\ref{2-3g} )in (\ref{2-3a}), we get
the following inequality%
\begin{eqnarray}
&&\left\vert \left( 1-\theta \right) \left( \lambda f(ma)+\left( 1-\lambda
\right) f(mb)\right) +\theta f(mC)-\frac{1}{m\left( b-a\right) }%
\dint\limits_{ma}^{mb}f(x)dx\right\vert  \notag \\
&\leq &m\left( b-a\right) \left( \frac{\theta ^{p+1}+\left( 1-\theta \right)
^{p+1}}{p+1}\right) ^{\frac{1}{p}}\left\{ \lambda ^{2}\left( \frac{%
\left\vert f^{\prime }\left( mC\right) \right\vert ^{q}+\alpha m\left\vert
f^{\prime }\left( a\right) \right\vert ^{q}}{\alpha +1}\right) ^{\frac{1}{q}%
}\right.  \notag \\
&&\left. \left( 1-\lambda \right) ^{2}\left( \frac{\left\vert f^{\prime
}\left( mC\right) \right\vert ^{q}+\alpha m\left\vert f^{\prime }\left(
b\right) \right\vert ^{q}}{\alpha +1}\right) ^{\frac{1}{q}}\right\} .
\label{2-3h}
\end{eqnarray}%
From the inequalities (\ref{2-3e}) and (\ref{2-3h}) \ the inequality (\ref%
{2-3}) is obtained.This completes the proof.
\end{proof}

\begin{corollary}
\label{2.5}Under the assumptions of Theorem \ref{2.4} with $\lambda =\frac{1%
}{2}$ and $\theta =\frac{2}{3}$, we have%
\begin{eqnarray*}
&&\left\vert \frac{1}{6}\left[ f(ma)+4f\left( \frac{ma+mb}{2}\right) +f(mb)%
\right] -\frac{1}{m\left( b-a\right) }\dint\limits_{ma}^{mb}f(x)dx\right\vert
\\
&\leq &\frac{m\left( b-a\right) }{12}\left( \frac{2^{p+1}+1}{3\left(
p+1\right) }\right) ^{\frac{1}{p}}\min \left\{ E_{1}^{\frac{1}{q}}(\frac{1}{2%
},\alpha )+E_{2}^{\frac{1}{q}}(\frac{1}{2},\alpha ),E_{3}^{\frac{1}{q}}(%
\frac{1}{2},\alpha )+E_{4}^{\frac{1}{q}}(\frac{1}{2},\alpha )\right\} ,
\end{eqnarray*}%
where%
\begin{eqnarray*}
E_{1}(\frac{1}{2},\alpha ) &=&\frac{\left\vert f^{\prime }\left( ma\right)
\right\vert ^{q}+\alpha m\left\vert f^{\prime }\left( \frac{a+b}{2}\right)
\right\vert ^{q}}{\alpha +1}, \\
E_{2}(\frac{1}{2},\alpha ) &=&\frac{\left\vert f^{\prime }\left( mb\right)
\right\vert ^{q}+\alpha m\left\vert f^{\prime }\left( \frac{a+b}{2}\right)
\right\vert ^{q}}{\alpha +1}, \\
E_{3}(\frac{1}{2},\alpha ) &=&\frac{\left\vert f^{\prime }\left( \frac{%
m\left( a+b\right) }{2}\right) \right\vert ^{q}+\alpha m\left\vert f^{\prime
}\left( a\right) \right\vert ^{q}}{\alpha +1}, \\
E_{4}(\frac{1}{2},\alpha ) &=&\frac{\left\vert f^{\prime }\left( \frac{%
m\left( a+b\right) }{2}\right) \right\vert ^{q}+\alpha m\left\vert f^{\prime
}\left( b\right) \right\vert ^{q}}{\alpha +1},
\end{eqnarray*}
\end{corollary}

\begin{remark}
In Corollary \ref{2.5}, if we take $\alpha =m=1,$ then we obtain the
following inequality%
\begin{eqnarray*}
&&\left\vert \frac{1}{6}\left[ f(a)+4f\left( \frac{a+b}{2}\right) +f(b)%
\right] -\frac{1}{b-a}\dint\limits_{a}^{b}f(x)dx\right\vert \leq \left( 
\frac{b-a}{12}\right) \left( \frac{1+2^{p+1}}{3\left( p+1\right) }\right) ^{%
\frac{1}{p}} \\
&&\times 2.\min \left\{ \left( \frac{\left\vert f^{\prime }\left( \frac{a+b}{%
2}\right) \right\vert ^{q}+\left\vert f^{\prime }\left( a\right) \right\vert
^{q}}{2}\right) ^{\frac{1}{q}},\left( \frac{\left\vert f^{\prime }\left( 
\frac{a+b}{2}\right) \right\vert ^{q}+\left\vert f^{\prime }\left( b\right)
\right\vert ^{q}}{2}\right) ^{\frac{1}{q}}\right\} ,
\end{eqnarray*}%
which is the better than the inequality in \cite[Corollary 3]{SSO10}
\end{remark}

\begin{corollary}
Under the assumptions of Theorem \ref{2.4} with $\lambda =\frac{1}{2}$ and $%
\theta =0$, we have%
\begin{eqnarray*}
&&\left\vert \frac{f(ma)+f(mb)}{2}-\frac{1}{m\left( b-a\right) }%
\dint\limits_{ma}^{mb}f(x)dx\right\vert \\
&\leq &\frac{m\left( b-a\right) }{4}\left( \frac{1}{p+1}\right) ^{\frac{1}{p}%
}\min \left\{ E_{1}^{\frac{1}{q}}(\frac{1}{2},\alpha )+E_{2}^{\frac{1}{q}}(%
\frac{1}{2},\alpha ),E_{3}^{\frac{1}{q}}(\frac{1}{2},\alpha )+E_{4}^{\frac{1%
}{q}}(\frac{1}{2},\alpha )\right\} .
\end{eqnarray*}
\end{corollary}

\begin{corollary}
Under the assumptions of Theorem \ref{2.4} with $\lambda =\frac{1}{2}$ and $%
\theta =1$, we have%
\begin{eqnarray*}
&&\left\vert f\left( \frac{m\left( a+b\right) }{2}\right) -\frac{1}{m\left(
b-a\right) }\dint\limits_{ma}^{mb}f(x)dx\right\vert \\
&\leq &\frac{m\left( b-a\right) }{4}\left( \frac{1}{p+1}\right) ^{\frac{1}{p}%
}\min \left\{ E_{1}^{\frac{1}{q}}(\frac{1}{2},\alpha )+E_{2}^{\frac{1}{q}}(%
\frac{1}{2},\alpha ),E_{3}^{\frac{1}{q}}(\frac{1}{2},\alpha )+E_{4}^{\frac{1%
}{q}}(\frac{1}{2},\alpha )\right\} .
\end{eqnarray*}
\end{corollary}

\end{document}